\pdfoutput=1
\documentclass[10pt]{article}

\usepackage[twoside,centering,a4paper,bottom=2.8cm]{geometry}

\usepackage{graphicx,color}

\usepackage{amssymb,amsmath,amsthm}

\theoremstyle{plain}
  \newtheorem{theorem}             {Theorem}    [section]
  \newtheorem{lemma}      [theorem]{Lemma}
  \newtheorem{corollary}  [theorem]{Corollary}
\theoremstyle{definition}
  
\renewcommand{\phi}{\varphi}                                                    
\renewcommand{\epsilon}{\varepsilon}                                            
\newcommand{\eps}{\varepsilon}                                                  

\DeclareMathOperator{\sign}{sign}

\newcommand{\R}{\mathbb{R}}

\newcommand{\wto}{\rightharpoonup}

\newcommand{\sm}{\setminus}

\newcommand{\C}{\mathcal{C}}

\newcommand{\E}{\mathcal{E}}
\newcommand{\F}{\mathcal{F}}

\renewcommand{\P}{\mathcal{P}}
\renewcommand{\S}{\mathbb{S}}

\numberwithin{figure}{section}
\numberwithin{table}{section}
\numberwithin{equation}{section}

\renewenvironment{abstract}{\footnotesize\begin{quote}\textbf{Abstract.}~}%
  {\end{quote}}

\title{Snapping elastic curves as a one-dimensional analogue of
  two-component lipid bilayers}

\author{%
  {\scshape Michael Helmers}\\%
  {\footnotesize Oxford Centre for Nonlinear PDE}\\[-1ex]
  {\footnotesize Mathematical Institute, University of Oxford}\\[-1ex]%
  {\footnotesize 24-29 St Giles', Oxford, OX1 3LB, United Kingdom}\\[-1ex]%
  {\footnotesize Email: {\tt helmers@maths.ox.ac.uk}}%
}

\date{}

\begin{document}

\maketitle

% -----------------------------------------------------------------------------
% -
% - Abstract
% -
% -----------------------------------------------------------------------------

\begin{abstract}
  In order to study a one-dimensional analogue of the spontaneous
  curvature model for two-component lipid bilayer membranes we
  consider planar curves that are made of a material with two
  phases. Each phase induces a preferred curvature to the curve, and
  these curvatures as well as phase boundaries may lead to the
  development of kinks. We introduce a family of energies for smooth
  curves and phase fields, and we show that these energies
  $\Gamma$-converge to an energy for curves with a finite number of
  kinks. The theoretical result is illustrated by some numerical
  examples.


\end{abstract}

% -----------------------------------------------------------------------------
% -
% - Introduction and Result
% -
% -----------------------------------------------------------------------------

\section{Introduction and Result}

Integral functionals depending on curvature are of geometric interest
and arise in a variety of applications such as image processing and
models for elastic lines or thin shells\cite{W93,S93,BM04,LS84}; in
particular, they appear in the study of biological
membranes.\cite{C70,H73,S97} In the spontaneous-curvature model for
bilayer vesicles with two lipid components equilibrium shapes are
described as surfaces minimising the energy
\begin{equation}
  \label{eq:Helfrich_energy}
  \sum_{i=1,2} \int_{\Sigma_i} a_i(H-C_i)^2 + b_i K \,dS
  + \sigma|\partial\Sigma_1|
\end{equation}
among all closed surfaces $\Sigma=\Sigma_1 \cup \Sigma_2$ with fixed
areas $|\Sigma_i|$ and fixed enclosed volume.\cite{BHW03,JL96} Here
$H$ and $K$ are the mean and Gauss curvature of the membrane surface
$\Sigma$, $a_i$ and $b_i$ are parameters related to bending resistance
of the membrane, and $\sigma$ is the line tension at the component
boundary $|\partial\Sigma_i|$; the spontaneous curvatures $C_i$ are
supposed to reflect an asymmetry in the membrane.

J\"ulicher and Lipowsky\cite{JL96} study the Euler-Lagrange equations
of \eqref{eq:Helfrich_energy} for axially symmetric membranes with
exactly one interface represented by a point on a rotated curve.  They
briefly discuss the possibility of different smoothness conditions for
the curve at the interface, their analysis, however, is done for
smooth membranes only. Du, Wang\cite{DW08} and Lowengrub, R{\"a}tz,
Voigt\cite{LRV09} perform numerical simulations using a phase field
for both the membrane and the lipid components; convergence to the
sharp interface model is obtained by asymptotic analysis or under
strong smoothness assumption on the limit surface.

In this paper we are interested in a one-dimensional analogue of the
spontaneous-curvature model for two component vesicles.  We consider
curves made of a material with two phases, each of which induces a
preferred bending to the curve; in contrast to the above studies for
membranes we do not enforce smoothness of the curves a priori. We
analyse an approximation by more regular curves, which, in particular,
can be treated numerically in an easier way than the model with kinks.

More precisely, we consider closed plane curves $q$ of fixed length $L$
that are twice weakly differentiable and regular except for a finite
number of points. These curves can be parametrised with unit speed
over the circle $\S^1$, when it is given a scaled standard metric to
have length $L$. We require that the squared mean curvature
$\kappa_q^2=|q''|^2$ of $q$ is integrable, so we let
\begin{align*}
  \C := \big\{ q \in C(\S^1;\R^2)
  &: \text{there exists a set } S_q
  \text{ of finitely many points s.\,t. }\\
  &\quad
  q \in H^2(\S^1 \sm S_q;\R^2), \,|q'|=1 \text{ in } \S^1 \sm S_q
  \big\}
\end{align*}
be the set of parametrised curves which may have a finite number of
kinks.  Indeed, because $H^2$ embeds continuously in $C^1$, $S_q$ is
the set of discontinuities of the tangent vector $q'$ or, as $|q'|$ is
constant, of the tangent angle. Note also that $\C \subset
H^1(\S^1;\R^2)$.

The material phases of $q \in \C$ are determined by a function
$v:\S^1 \to \{\pm1\}$ having at most a finite number of jumps
and satisfying the volume constraint
\begin{equation}
  \label{eq:volume_constraint}
  \int_{\S^1} v \,dt = m L
\end{equation}
for fixed $m\in(-1,1)$. We denote the set of such functions by $\P$
and the jump set of $v\in\P$ by $S_v$; note that $\P \subset
BV(\S^1;\{\pm1\})$.

On basis of the membrane model we consider for $(q,v) \in \C\times\P$
the energy
\begin{equation}
  \label{eq:energy}
  \E(q,v)
  := \int_{\S^1 \sm (S_q \cup S_v)}
  \left(\kappa_q-C(v)\right)^2 dt
  + \sum_{s\in S_q \cup S_v} \left( \sigma + \hat\sigma |[q'](s)| \right).
\end{equation}
Compared to \eqref{eq:Helfrich_energy} we have dropped the Gauss
curvature term, as curves have no intrinsic curvature. Furthermore,
for notational simplicity we have set all bending rigidities to one
and let only the spontaneous curvatures $C(\pm1)$ be
phase-dependent. Different rigidities can be treated similar to the
spontaneous curvature below.

In \eqref{eq:energy} interfaces without kinks are penalised by the
constant energy $\sigma$, while kinks carry an additional ``bending
energy'' $\hat\sigma|[q']|(s)$ where $\hat\sigma$ is a constant and
$|[q'](s)|$ denotes the modulus of the angle enclosed by the two
one-sided tangent vectors at $s$ modulo $2\pi$.  Note, that kinks may
not only occur at interfaces, but also within a phase. Such kinks can
be seen as resembling budding transitions or non-smooth limit shapes
of even single-component membranes; we shall call them \emph{ghost
  interfaces}.

As an approximation to this model we consider curves from the set
\begin{equation*}
  \C_\eps
  := \big\{ q \in H^2(\S^1;\R^2):
  \, |q'|=1 \text{ in } \S^1 \big\}.
\end{equation*}
We replace sharp material phases by phase fields $v \in H^1(\S^1)$
with the constraint \eqref{eq:volume_constraint}, denoting this set of
functions by $\P_\eps$.
For $\eps>0$ and $(q,v) \in \C_\eps \times \P_\eps$ we consider the
energy
\begin{equation}
  \label{eq:eps-energy}
  \E_\eps(q,v)
  := \int_{\S^1} v^2 \left(\kappa_q-C(v)\right)^2 dt
  + \int_{\S^1} \eps v'^2 + \tfrac{1}{\eps} \Phi(v) \,dt
  + \eps \int_{\S^1} \kappa_q^2 \,dt,
\end{equation}
where $\Phi:\R \to [0,\infty)$ is a continuous double-well potential,
that is zero only in $\pm1$ and satisfies $\Phi(v)\to\infty$ as
$v\to\pm\infty$. For notational simplicity we assume that $\Phi$ is
symmetric with respect to the origin, and for technical reasons that
it is $C^2$ in a neighbourhood of its minima.  The function
$C:\R\to\R$ is a continuous and bounded extension of $C(\pm1)$.

The first integral in \eqref{eq:eps-energy} resembles the curvature
integral in \eqref{eq:energy}; having the phase field in front of the
curvature term enables the curves to approach kinks as $\eps\to0$.
The third integral, however, penalises regions of very large curvature
and accounts for a kink's bending energy in the limit.  Finally,
interface costs are contributed by the second integral.

Below we show that the $\eps$-energies \eqref{eq:eps-energy} converge
to \eqref{eq:energy} with
\begin{equation}
  \label{eq:constants}
  \sigma = 2\int_{-1}^{1} \sqrt{\Phi(v)} \,dv
  \qquad\text{and}\qquad
  \hat\sigma = 2\sqrt{\Phi(0)}.
\end{equation}
In order to formulate and prove our theorem we fix these constants as
in \eqref{eq:constants}. We extend the energies $\E_\eps$ and $\E$
to the space $H^1(\S^1;\R^2) \times L^1(\S^1)$ by setting
$\E_\eps(q,v)=\E(q,v)=\infty$ whenever $(q,v)$ does not belong to
$\C_\eps \times \P_\eps$ and $\C \times \P$, respectively.

\begin{theorem}
  \label{main_theorem}
  The energies $\E_\eps$ are equi-coercive, that is any sequence
  $(q_\eps,v_\eps) \in \C_\eps \times \P_\eps$ with uniformly bounded
  energy admits a subsequence converging strongly in $H^1(\S^1;\R^2)
  \times L^1(\S^1)$ to some $(q,v) \in \C \times \P$.

  Furthermore, the $\E_\eps$ $\Gamma$-converge to $\E$ as $\eps \to
  0$, that is
  \begin{itemize}
  \item for any sequence $(q_\eps,v_\eps)$ that converges to some
    $(q,v)$ in $H^1(\S^1;\R^2) \times L^1(\S^1)$ as $\eps\to0$ we have
    \begin{equation*}
      \liminf_{\eps\to0} \E_\eps(q_\eps,v_\eps) \geq \E(q,v);
    \end{equation*}
  \item for any $(q,v)$ there is a sequence $(q_\eps,v_\eps)$
    converging to $(q,v)$ in $H^1(\S^1;\R^2) \times L^1(\S^1)$ such
    that
    \begin{equation*}
      \limsup_{\eps\to0} \E_\eps(q_\eps,v_\eps) \leq \E(q,v).
    \end{equation*}
  \end{itemize}
\end{theorem}

We postpone the proof of Theorem \ref{main_theorem} to Section 2 in
favour of some remarks and illustrations. Two examples of a local
minimiser of $\E_\eps$ are given in Figure \ref{fig:result1} for
$\eps=0.05$ and $C(v)$ being a cubic interpolation of $C(-1)=1$,
$C(+1)=2$, $C'(\pm1)=0$; the potential, and therewith the cost of
kinks and interfaces, is $\Phi(v)=(1-v^2)^2$ for the left pictures and
$\Phi(v)=0.75(1-v^2)^2$ for the right. Both results are obtained by a
gradient flow for $\E_\eps$ with respect to the $H^{-1}$ norm for the
phase field and the $L^2$ norm for the tangent angle; see
\cite{H09,EFM89,EF87} for details.

In both simulations the initial curve is a circle of radius $2$ and
the initial phase field has mean value zero with two interface
regions. The interfaces are retained during the evolution, but new
small areas of large curvature appear within the phase of spontaneous
curvature $2$. As already mentioned, these additional regions may
persist as $\eps$ tends to zero, giving rise to ghost interfaces.
Between (ghost) interfaces the numerically stationary curve consists
of segments of circles whose curvatures are determined by the phase,
but not equal to the preferred ones.

Our second note is the existence of minimisers for $\E$. From the
properties of $\Gamma$-convergence and equi-coercivity, see for
instance \cite{B02}, we know that any sequence
$(q_\eps,v_\eps) \in \C_\eps \times \P_\eps$ satisfying
\begin{equation*}
  \E_\eps(q_\eps,v_\eps) = \inf_{\C_\eps \times \P_\eps}
  \E_\eps + o(1),
\end{equation*}
admits a subsequence converging to a minimiser $(q,v)$ of $\E$ in $\C
\times \P$. As the energy \eqref{eq:eps-energy} is bounded from below,
we can always find such almost minimising sequences.  By the Direct
Method of the Calculus of Variations there exists even a minimiser for
each $\eps>0$, because $\E_\eps$ is coercive and weakly lower
semi-continuous on $H^2(\S^1;\R^2) \times H^1(\S^1)$, and $\C_\eps
\times \P_\eps$ is nonempty and weakly closed.

Finally, let us discuss three extensions of our theorem.  First of
all, the proof is easily adapted to non-symmetric potentials
$\Phi$. In this case one splits $\sigma$ into two constants
$\sigma^\pm$, defined as the integral of $\Phi$ over $(0,1)$ and
$(-1,0)$, and distinguishes proper interfaces and ghost interfaces in
different phases by their constant energy contribution $\sigma^+ +
\sigma^-$, $2\sigma^+$ or $2\sigma^-$. One may also consider
potentials like $\Phi(v) = (1-v)^2$ and drop the volume constraint for
$v_\eps$. Then there is only one material phase, and the $v_\eps$ are
mere auxiliary variables to allow curvature induced kinks.

Second, changing the power of $\eps$ in the last term of
\eqref{eq:eps-energy} to $\eps^k$, $k>1$, or even dropping the term
completely yields the $\Gamma$-limit \eqref{eq:energy} with
$\hat\sigma=0$, that is, without bending contribution of kinks; the
underlying topology changes to weak $H^1$ convergence of the curves.

Third, the arguments can be extended to handle non-closed curves with
fixed end points and prescribed tangents in the approximate model. The
additional issue in this situation is that kinks may appear at the
boundary in the sense that the tangent vector of the limit curve
differs from the prescribed one. As for ghost interfaces this yields a
contribution to the limit energy; see \cite{H09} for the details.

\begin{figure}
  \centering
  \includegraphics[width=\textwidth]{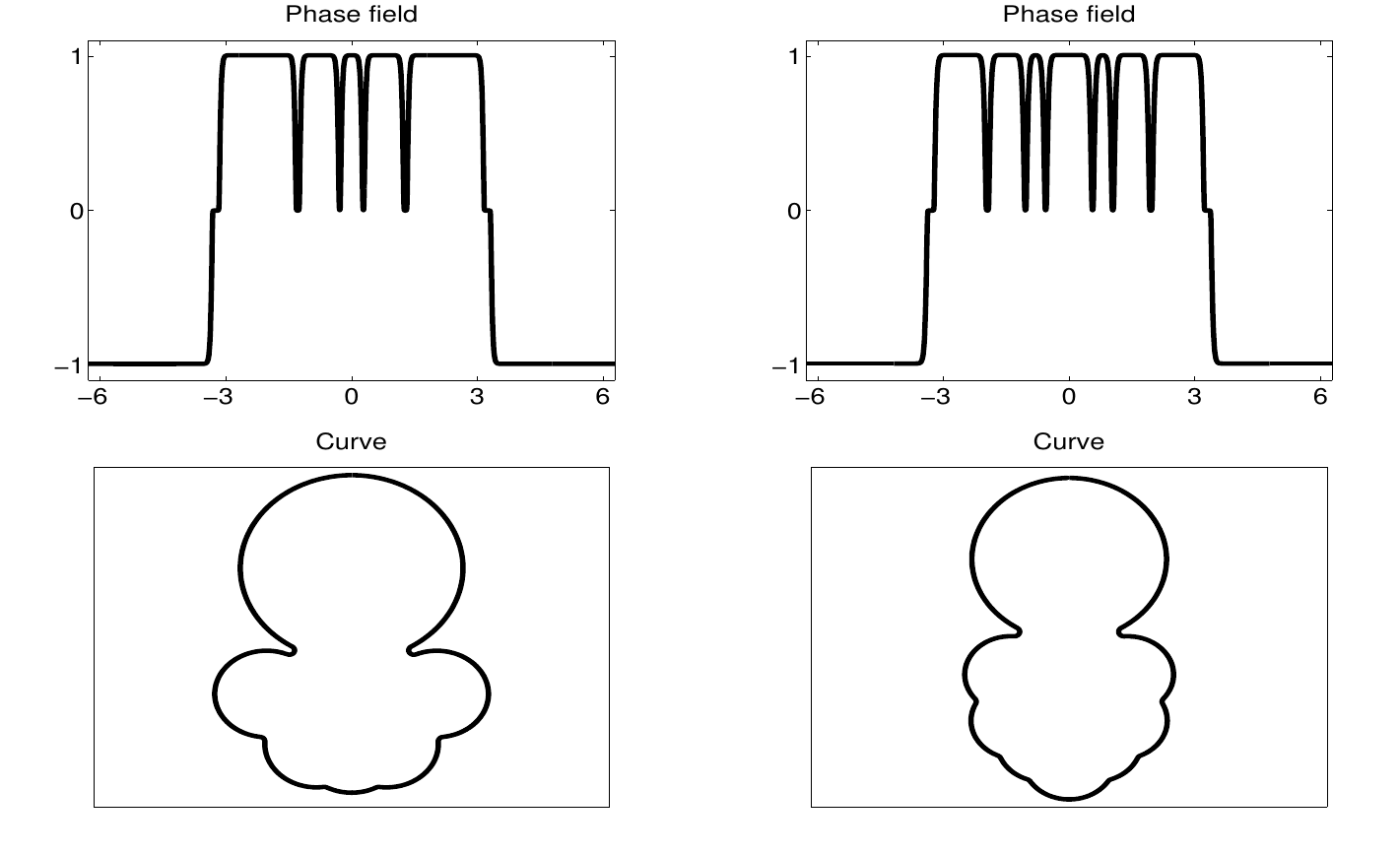}
  \caption{Two examples of numerically local minimisers of $\E_\eps$.
    In the upper figures the phase fields are plotted over the sphere
    parametrised by the interval $[-2\pi,2\pi]$; the lower figures
    show the curves in the $xy$-plane.}
  \label{fig:result1}
\end{figure}

% -----------------------------------------------------------------------------
% -
% - Proof
% -
% -----------------------------------------------------------------------------

\section{Proof of Theorem \ref{main_theorem}}

In this section we present the proof of Theorem \ref{main_theorem}. We
first show equi-coercivity of the energies $\E_\eps$, then establish
the lower bound inequality and close with the upper bound.  An
important ingredient in what follows is the fact that given
$q\in\C_\eps$ and a directed line in the plane there is $u \in
H^1(\S^1)$ such that $u(t)$ is the angle between $q'(t)$ and this
line; $u$ is uniquely determined up to adding multiples of $2\pi$, and
in addition we have $\kappa_q(t) = u'(t)$ for all $t \in \S^1$.  On
the other hand, the curve is uniquely determined by fixing one point
together with its tangent there and an angle function.  For $q \in \C$
we can still find an angle function $u \in H^1(\S^1 \sm S_q)$, but as
jumps can be arbitrarily large it is not unique anymore; of course we
can assume that each jump is less than $2\pi$.  If we are only
interested in finding a local angle function near a kink, its jump can
be bounded by the enclosed angle of the limit tangents if the line is
chosen appropriately.  In such a local setting $|[q']|$ is simply
given by $|[u]|$.

% -----------------------------------------------------------------------------
% - Equi-coercivity
% -----------------------------------------------------------------------------

\subsection{Equi-coercivity}

\begin{lemma}
  Let $(q_\eps,v_\eps) \in \C_\eps \times \P_\eps$ be a sequence
  with uniformly bounded energy $\E_\eps(q_\eps,v_\eps)$.  Then there
  are $(q,v) \in \C \times \P$ and a subsequence,
  not relabelled, such that $q_\eps \to q$ in $H^1(\S^1;\R^2)$ and
  $v_\eps \to v$ in $L^1(\S^1)$.

  Furthermore, for this subsequence there are global angle functions
  $u_\eps \in H^1(\S^1)$ that converge weakly in $BV(\S^1)$ to an
  angle function $u$ of $q$.
\end{lemma}

\begin{proof}
  The argument for the sequence of phase fields is based on well-known
  observations by Modica and Mortola,\cite{M87,MM77} see in particular
  \cite{B02} for a proof in one dimension.  The outcome is a
  finite set of points $\tilde S \subset \S^1$ and a function $v \in
  \P$ whose jump set $S_v$ is contained in $\tilde S$ such that a
  subsequence $v_\eps$ converges to $v$ in measure and pointwise on
  $\S^1 \sm \tilde S$. Moreover, in the one-dimensional setting
  $(v_\eps)$ is uniformly bounded in $L^\infty(\S^1)$, hence the
  subsequence converges in $L^p(\S^1)$ for any $p<\infty$; we also
  have $|v_\eps| \geq 1/2$ for sufficiently small $\eps$ in any
  interval compactly contained in $\S^1 \sm \tilde S$.

  Taking into account only the just selected subsequence, we address
  the curves. As translations and rotations do not change the energy,
  we may assume that all curves pass at a fixed $s_0\in\S^1$ through a
  common point with the same tangent vector $\tau_0$.  From this and
  the fact that $|q_\eps'|=1$ we get $q_\eps \wto q$ in
  $H^1(\S^1;\R^2)$ for a subsequence.
  To show that $q \in H^2(\S^1\sm \tilde S)$ let $I$ be open and
  compactly contained in $\S^1 \sm \tilde S$. As $v_\eps^2 \geq 1/4$
  in $I$ for sufficiently small $\eps$, the sequence
  $(\kappa_{q_\eps}^2)=(|q''_\eps|^2)$ is bounded in $L^1(I)$; thus a
  subsequence of $(q''_\eps)$ converges weakly in $L^2(I;\R^2)$ to
  some $q''_I$. But then $q_\eps$ converges weakly in $H^2(I;\R^2)$
  and from uniqueness of the weak limit we infer that $q''_I$ is the
  weak derivative of $q'$ in $I$ and that the whole sequence
  converges.  This convergence combined with $|v_\eps|\to1$ and
  $\sup_\eps \|v_\eps\|_{\infty} < \infty$ yields the estimate
  \begin{equation}
    \label{eq:liminf_compactness}
    \int_I |q''_I|^2 \,dt
    \leq \liminf_{\eps \to 0}
    \int_I v_\eps^2 |q''_\eps|^2 \,dt
    \leq \liminf_{\eps \to 0}
    \int_{\S^1} v_\eps^2 |q''_\eps|^2 \,dt,
  \end{equation}
  where the right hand side is bounded by the energy.  Since
  \eqref{eq:liminf_compactness} is true for any $I \Subset\S^1 \sm
  \tilde S$, $q''$, defined as $q''_I$ on $I \Subset \S^1
  \sm \tilde S$, belongs to $L^2(\S^1;\R^2)$, and $q \in H^2(\S^1 \sm
  \tilde S; \R^2)$ follows.

  It remains to establish convergence of angle functions and to
  improve the convergence of the curves. Applying H\"older's
  inequality we get
  \begin{align*}
    \int_{\S^1} |\kappa_{q_\eps}| \,dt & \leq \int_{\{|v_\eps|<1/2\}}
    |\kappa_{q_\eps}| \,dt + \int_{\{|v_\eps|\geq1/2\}}
    |\kappa_{q_\eps}| \,dt
    \\
    & \leq \left( \int_{\S^1} \eps \kappa_{q_\eps}^2 \,dt
    \right)^{1/2} \left( \tfrac{1}{\eps} |\{|v_\eps|<1/2\}|
    \right)^{1/2} + 2 \sqrt{L}
    \left( \int_{\S^1} v_\eps^2\kappa_{q_\eps}^2 \,dt \right)^{1/2}.
  \end{align*}
  Here the curvature integrals are bounded by $\E(q_\eps,v_\eps)$, and
  since $\Phi$ has a positive minimum on $[-1/2,1/2]$, the quantity
  $|\{|v_\eps|<1/2\}|/\eps$ is controlled by the potential energy.
  Hence, with $\bar u$ satisfying $(\cos \bar u, \sin \bar u) =
  \tau_0$, the global angle functions
  \begin{equation*}
    u_\eps(s) = \bar u + \int_{s_0}^s \kappa_{q_\eps}(t) \,dt
  \end{equation*}
  are uniformly bounded in $L^\infty(\S^1)$ and
  $W^{1,1}(\S^1\sm\{s_0\})$.  Therefore there is a subsequence such
  that $u_\eps \to u$ almost everywhere and weakly in $BV(\S^1)$, that
  is $u_\eps \to u$ in $L^1(\S^1)$ and $\kappa_{q_\eps} dt$ weakly to
  the measure $Du$. Consequently, $q'_\eps = (\cos u_\eps, \sin
  u_\eps) \to (\cos u, \sin u) = q'$ in $L^2(\S^1;\R^2)$ and $q_\eps
  \to q$ in $H^1(\S^1;\R^2)$.
\end{proof}

% -----------------------------------------------------------------------------
% - Liminf Inequality
% -----------------------------------------------------------------------------

\subsection{Lower bound inequality}

Next we prove the lower bound inequality
\begin{equation*}
  \liminf_{\eps\to0} \E_\eps(q_\eps,v_\eps) \geq \E(q,v)
\end{equation*}
whenever $(q_\eps,v_\eps)$ converges to $(q,v)$ in $H^1(\S^1;\R^2)
\times L^1(\S^1)$. It suffices to examine the case when
$(q_\eps,v_\eps) \in \C_\eps \times \P_\eps$ and to consider a
subsequence such that the lower limit is attained. Then our
compactness argument shows that $(q,v) \in \C \times \P$ and $S_q \cup
S_v \subset \tilde S$, where $\tilde S \subset \S^1$ is a finite set
of points. The same arguments as in \eqref{eq:liminf_compactness} and
the convergence $C(v_\eps) \to C(v)$ in $L^2(\S^1)$ yield
\begin{equation}
  \label{eq:liminf_curvature}
  \int_{\S^1 \sm (S_q \cup S_v)} \left( \kappa_q - C(v) \right)^2 \,dt
  \leq
  \liminf_{\eps\to0} \int_{\S^1} v_\eps^2 \left( \kappa_{q_\eps} - C(v_\eps)
    \right)^2 \,dt.
\end{equation}
As points in $\tilde S \sm (S_q \cup S_v)$ do not contribute to the
limit energy, the task is to understand what happens near kinks and
interfaces.  To this end it is convenient to introduce the
set-dependent energies
\begin{align*}
  \F_\eps(q_\eps,v_\eps,I) &= \int_I \eps v_\eps'^2
  + \tfrac{1}{\eps} \Phi(v_\eps) + \eps \kappa_{q_\eps}^2 \,dt
\\
%  \quad\text{and}\quad
\intertext{and}
  \F(q,v,I) &= \sum_{s \in (S_q \cup S_v) \cap I} \left(
    \sigma + \hat\sigma |[q'](s)| \right)
\end{align*}
for $I \subset \S^1$.  In what follows we establish estimates for
$\F_\eps$ and $\F$ in the case of an interface without a kink, extend
the argument to interfaces with a kink, and afterwards deal with ghost
interfaces.  The inequality for $\E_\eps$ and $\E$ then follows by
combining these estimates with \eqref{eq:liminf_curvature}.

% - Case 1: jump in phase field, no jump in curve tangent ---------------------

\subsubsection{Interfaces without kink: $s \in S_v \sm S_q$}

Let $I$ be an open interval in $\S^1$ such that $\bar I \cap \tilde S
= \{s\}$. As the curve $q$ has no kink in $I$, it does not contribute
to the limit energy $\F(q,v,I)$, so we estimate the curvature term of
$\F_\eps(q_\eps,v_\eps,I)$ simply by zero.  The lower bound of the
remaining part
\begin{equation}
  \label{ModicaMortola}
  \liminf_{\eps\to0}
  \int_I \eps v'^2_\eps + \tfrac{1}{\eps} \Phi(v_\eps) \,dt
  \geq
  \sigma
\end{equation}
is the well-known result for phase transitions by Modica and
Mortola.\cite{MM77,M87} In fact, there are points $a_\eps,b_\eps \in
I$, $a_\eps<s<b_\eps$ or $b_\eps<s<a_\eps$ such that $v_\eps(a_\eps)
\to -1$ and $v_\eps(b_\eps) \to 1$ for a subsequence $\eps\to0$;
restricting the integral to $(a_\eps,b_\eps)$ or $(b_\eps,a_\eps)$,
inequality \eqref{ModicaMortola} follows from Young's inequality and a
substitution of variables.

% - Case 2: jump in phase field, jump in curve tangent ------------------------

\subsubsection{Interfaces with kink: $s \in S_v \cap S_q$}

Now let $s$ be a point where the curve $q$ has a kink and fix
local angle functions $u_\eps$ in a small interval $I$ around $s$ such
that $\bar I$ contains no other point of $\tilde S$.  By our
equi-coercivity result we may assume that $(u_\eps)$ converges weakly
in $BV(I)$ to an angle function $u$ of $q$; in particular, we have
\begin{equation*}
  \int_I u_\eps' \,dt \to [u](s) + \int_I \kappa_q \,dt.
\end{equation*}
Note that $|[u](s)| \geq |[q'](s)|$ with strict inequality possible if
the curves $q_\eps$ have loops near $s$ that vanish in the limit.

We split the neighbourhood $I$ of $s$ into two parts: one where
$v_\eps$ is close to zero and the other where its transition to $\pm1$
takes place; we expect $q_\eps$ to approximate the kink in the former
part.

\begin{lemma}
  \label{M_eps_delta}
  For $I$, $u$ as above and any sufficiently small $\delta>0$ let
  \begin{equation*}
    M_{\eps,\delta} = \{ t \in \S^1 : |v_\eps(t)| \leq \delta \}
  \end{equation*}
  be the set where $|v_\eps|$ is bounded by $\delta$. Then
  \begin{equation*}
    \liminf_{\eps\to0} \left|
      \int_{I \cap M_{\eps,\delta}} u_\eps' \,dt \right|
    \geq
    |[u](s)|.
  \end{equation*}
\end{lemma}

\begin{proof}
  Let $\gamma>0$ be arbitrary but fixed, and let
  $U_\gamma:=[s-\gamma,s+\gamma]$. As $I \sm U_\gamma$ is compactly
  contained in $\S^1 \sm \tilde S$ we have $|v_\eps|\geq2\delta$ in $I
  \sm U_\gamma$ for all sufficiently small $\eps$, and therefore $I
  \cap M_{\eps,\delta} \subset U_\gamma$.  Writing
  $w_\eps=u'_\eps-\kappa_q$,
  we have
  \begin{equation}
    \label{eq:bd}
    \left| \int_{I \sm M_{\eps,\delta}} w_\eps \,dt \right|
    \leq
    \left| \int_{I \sm U_\gamma} w_\eps \,dt \right|
    + \int_{(I \sm M_{\eps,\delta}) \cap U_\gamma} |w_\eps| \,dt
  \end{equation}
  for all sufficiently small $\eps$.  The first term in \eqref{eq:bd}
  converges to zero as $\eps\to0$ by weak convergence of $w_\eps dt$
  in $I \sm U_\gamma$, and the second is less than a constant times
  $\sqrt{\gamma}$ due to H\"older's inequality and the energy bound.
  As $\gamma>0$ is arbitrary we obtain
  \begin{equation*}
    \limsup_{\eps\to0} \left| \int_{I \sm M_{\eps,\delta}} w_\eps \,dt \right|
    = 0,
  \end{equation*}
  and taking the lower limit in the inequality
  \begin{equation*}
    \left| \int_{I \cap M_{\eps,\delta}} w_\eps \,dt \right|
    \geq
    \left| \int_{I} w_\eps \,dt \right|
    - \left|\int_{I \sm M_{\eps,\delta}} w_\eps \,dt \right|
  \end{equation*}
  yields the claim as $\kappa_q\in L^2(I)$ and $|I \cap
  M_{\eps,\delta}|\to0$.
\end{proof}

Next we prove the key estimate for the lower bound inequality at kinks.

\begin{lemma}
  \label{JumpArgument}
  Let $I \subset \S^1$ be an open interval such that $\bar I$ contains
  exactly one point $s \in S_v \cap S_q$ and no other points of
  $\tilde S$. Then
  \begin{equation*}
    \liminf_{\eps\to0} \F_\eps(q_\eps,v_\eps,I)
    \geq
    \hat\sigma |[q'](s)|  + \sigma.
  \end{equation*}
\end{lemma}

\begin{proof}
  With the notation of Lemma \ref{M_eps_delta} we have
  \begin{equation}
    \label{eq:first_estimate}
    \F_\eps(q_\eps,v_\eps,I)
    \geq \int_{I \cap M_{\eps,\delta}} \eps u'^2_\eps
    + \tfrac{1}{\eps} \Phi(v_\eps) \,dt
    + \int_{I \sm M_{\eps,\delta}} \eps v'^2_\eps
    + \tfrac{1}{\eps} \Phi(v_\eps) \,dt.
  \end{equation}
  Estimating the first term with H\"older's and Young's inequality
  we get
  \begin{align*}
    \int_{I \cap M_{\eps,\delta}} \eps u'^2_\eps
    + \tfrac{1}{\eps} \Phi(v_\eps) \,dt
    & \geq \tfrac{\eps}{|I \cap M_{\eps,\delta}|}
    \left( \int_{I \cap M_{\eps,\delta}} u'_\eps \,dt\right)^2
    + \tfrac{|I \cap M_{\eps,\delta}|}{\eps}
    \inf_{v\in[-\delta,\delta]} \Phi(v) \\
    & \geq 2 \left| \int_{I \cap M_{\eps,\delta}} u'_\eps \,dt \right|
    \sqrt{\inf_{v\in[-\delta,\delta]} \Phi(v)}.
  \end{align*}
  With the second integral in \eqref{eq:first_estimate} we deal as in
  the case before; the only difference is that we now find an interval
  $(a_\eps,b_\eps) \subset I \sm M_{\eps,\delta}$ such that
  $v_\eps(a_\eps) \to \delta$, $v_\eps(b_\eps) \to 1$ on one side of
  $s$, and the same with $-\delta$ and $-1$ on the other. Putting both
  estimates together and passing to the lower limit as $\eps\to0$
  yields
  \begin{equation*}
    \liminf_{\eps\to0} \F_\eps(q_\eps,v_\eps,I)
    \geq
    2 |[u](s)| \sqrt{\inf_{v \in [-\delta,\delta]} \Phi(v)}
    + 2 \int_\delta^1 \sqrt{\Phi(v)} \,dv
    + 2 \int^{-\delta}_{-1} \sqrt{\Phi(v)} \,dv,
  \end{equation*}
  and taking the supremum over all $\delta>0$ completes the proof.
\end{proof}

% - Case 3: no jump in phase field, jump in curve tangent (ghost interface) ---

\subsubsection{Ghost interfaces: $s \in S_q \sm S_v$}

Finally, let $s \in S_q \sm S_v$ and $I \subset \S^1$ such that $\bar
I \cap \tilde S = \{s\}$ and the phase field $v$ is constant in $\bar
I$, say $v \equiv 1$. Then we argue as in Lemma \ref{M_eps_delta} and
\ref{JumpArgument} to find
\begin{equation*}
  \liminf_{\eps\to0} \F_\eps(q_\eps,v_\eps,I)
  \geq
  \hat\sigma |[q'](s)| + 2 \cdot 2 \int_{0}^{1} \sqrt{\Phi(v)} \,dv,
\end{equation*}
where the right-hand side is equal to $\hat\sigma |[q'](s)| + \sigma$
due to the symmetry of $\Phi$. A similar argument is true when $v
\equiv -1$ near $s$, and this concludes the proof of the lower bound
estimate for $\F_\eps$ and therewith $\E_\eps$.

% -----------------------------------------------------------------------------
% - Recovery
% -----------------------------------------------------------------------------

\subsection{Upper bound inequality}

The final subsection is devoted to the upper bound inequality. Given
$(q,v)$ with finite energy $\E(q,v)$, we find a recovery sequence
$(q_\eps,v_\eps)$ by changing $(q,v)$ around (ghost) interfaces.  For
each $s \in S=S_q \cup S_v$ we choose two nested intervals of size of
order $\eps$: In the inner the kink is smoothed out by a linear
interpolation of a local angle function, and in the outer the phase
field transition to $\pm1$ is made.
But we have to ensure not to violate any constraint and
not to tear the curve apart when applying local changes to $q$ and
$v$.

% - Curve recovery ------------------------------------------------------------

\subsubsection{The curve}

Let $s \in S_q$ and $I \subset \S^1$ with $\bar I \cap S_q = \{s\}$.
For simplicity of notation we identify points in $I$ with coordinates
that map $s$ to the origin and formulate the following arguments for
curves and phase fields given in an interval $I$ around $s=0$.

We fix a line passing through the kink so that the tangents $q'(t)$ as
$t\to0$ from above and below meet it with angle $\bar u$ and $-\bar u$
for some $\bar u \in(0,\pi/2]$; then the kink carries the ``bending
energy'' $2\bar u \hat\sigma$.  The local angle function $u$
corresponding to the line is uniformly continuous on either side of
$t=0$, hence by decreasing $I$ we may assume that $|u|<\pi$ in $I$ and
that $u(t)$ is negative for $t<0$ and positive for $t>0$.

In the simple case that $q$ is made up of two straight lines so that
$u$ is constant on either side of zero, the linear interpolation
$u_\eps$ on an interval $I_\eps=(-\delta_\eps,\delta_\eps) \subset I$
is given by
\begin{equation*}
  u_\eps(t) =
  \begin{cases}
    - \bar u                     & : t < - \delta_\eps, \\
    \frac{\bar u}{\delta_\eps} t & : |t| \leq \delta_\eps, \\
    \bar u                       & : t > \delta_\eps,
  \end{cases}
\end{equation*}
where $\delta_\eps$ is intended to go to zero as $\eps$ does.  For the
curve $q_\eps$, defined by $u_\eps$ and one endpoint of $q_\eps(I)$
equal to the corresponding endpoint of $q(I)$, we compute
\begin{equation*}
  \F_\eps(q_\eps,0,I_\eps)
  = 2 \frac{\eps}{\delta_\eps} \bar u^2 + 2 \frac{\delta_\eps}{\eps} \Phi(0)
  \geq 4 \bar u \sqrt{\Phi(0)}
  = \hat\sigma |[q']|,
\end{equation*}
using Young's inequality, and equality holds if and only if
\begin{equation*}
  \delta_\eps = \frac{\bar u}{\sqrt{\Phi(0)}} \eps
  = \frac{|[q']|}{2\sqrt{\Phi(0)}} \eps.
\end{equation*}
With this $\delta_\eps$ we return to the general case: the linear
interpolation of the angle on $I_\eps$ now is
\begin{equation*}
  u_\eps(t) =
  \begin{cases}
    u(t)
    &: \delta_\eps < |t|, \\
    \frac{\left( u(\delta_\eps)-u(-\delta_\eps) \right)}{2\delta_\eps} t +
    \frac{\left( u(\delta_\eps)+u(-\delta_\eps) \right)}{2}
    &: \delta_\eps \geq |t|,
  \end{cases}
\end{equation*}
and similarly as above we get
\begin{align*}
  \F_\eps(q_\eps,0,I_\eps)
  &= \frac{|u(\delta_\eps)-u(-\delta_\eps)|^2}{2 \bar u} \sqrt{\Phi(0)} +
  2 \bar u \sqrt{\Phi(0)} \\
  &\to 4 \bar u \sqrt{\Phi(0)}
  = \hat\sigma |[q']|.
\end{align*}

But as noted before, just replacing $q$ by $q_\eps$ on $I$ is not
admissible since the second endpoint of $q(I)$ is not reached by
$q_\eps(I)$ and the whole curve would become discontinuous.  Recalling
the relation of tangent and angle function, the condition on the
endpoints can be expressed as condition for $u_\eps$ by requiring
\begin{equation}
  \label{eq:constraints}
  \begin{split}
    \int_I \cos u_\eps(t) \,dt
    &= \int_I \cos u(t) \,dt
    =: C_0,\\
    \int_I \sin u_\eps(t) \,dt
    &= \int_I \sin u(t) \,dt
    =: S_0.
  \end{split}
\end{equation}
We amend the linear interpolation $u_\eps$ by adding a perturbation
that on the one hand is sufficiently small not to contribute to the
energy in the limit $\eps\to0$, but on the other hand corrects the
defect in the constraints \eqref{eq:constraints}. We will find two
smooth functions $f$ and $g$, which depend on $u$ in $I$, and two
parameters $\alpha_\eps$ and $\beta_\eps$ such that $u_\eps +
\alpha_\eps f + \beta_\eps g$ is admissible for sufficiently small
$\eps$; the argument is simply the Implicit Function Theorem applied
to
\begin{equation*}
  P(\eps,\alpha,\beta):=
  \begin{pmatrix}
    C_0 -
    \int_I \cos \left(
      u_\eps + \alpha f + \beta g \right) dt \\
    -S_0 +
    \int_I \sin \left(
      u_\eps + \alpha f + \beta g \right) dt \\
  \end{pmatrix}.
\end{equation*}

\begin{lemma}
  Let $q$, $u$ and $u_\eps$ be as above. There exist two functions
  $f,g \in C^\infty_0(I)$ such that there are $\eps_0>0$ and functions
  $\eps \mapsto \alpha_\eps$, $\eps \mapsto \beta_\eps$, continuously
  differentiable in $[0,\eps_0)$ that satisfy
  $P(\eps,\alpha_\eps,\beta_\eps) = 0$ for all $\eps \in [0,\eps_0)$.
\end{lemma}

\begin{proof}
  Writing $u$ as sum of a continuous function and a piecewise constant
  jump function and $u_\eps$ correspondingly, it is easily seen that
  $P$ is a $C^1$ function for sufficiently small $\eps \geq 0$.  To
  apply the Implicit Function Theorem we have to show that
  $\partial_{(\alpha,\beta)} P(0,0,0)$ is non-singular. To this end we
  define two linear continuous functionals $T_s,T_c: C^\infty_0(I) \to
  \R$,
  \begin{equation*}
    T_s \phi = \int_I \phi(t) \sin u(t) \,dt
    \qquad \text{and} \qquad
    T_c \phi = \int_I \phi(t) \cos u(t) \,dt,
  \end{equation*}
  and compute
  \begin{equation*}
    \partial_{(\alpha,\beta)} P(0,0,0) =
    \begin{pmatrix}
      T_s f & T_s g \\
      T_c f & T_c g
    \end{pmatrix}.
  \end{equation*}

  Assume for the moment that neither $T_s$ nor $T_c$ is constantly
  zero.
  Suppose for contradiction that
  $\ker T_s = \ker T_c$, which implies the
  existence of $\lambda\not=0$ such that $T_s = \lambda T_c$; hence
  $\sin u = \lambda \cos u$ in $I$. This can only be true if $u$ is
  piecewise constant, but then $\sin \bar u = \lambda \cos \bar u$ and
  $-\sin \bar u = \lambda \cos \bar u$ contradict each other due to
  $\lambda\not=0$.

  Thus $\ker T_s \not= \ker T_c$, say $\ker T_s \cap \ker T_c
  \subsetneq \ker T_c$, and there is $f \in \ker T_c$ such that $T_s
  f=1$. After fixing any $g$ with $T_c g=1$, the partial derivative is
  \begin{equation*}
    \partial_{(\alpha,\beta)} P(0,0,0) =
    \begin{pmatrix}
      1 & * \\
      0 & 1
    \end{pmatrix},
  \end{equation*}
  where $*$ is some real number. This matrix is non-singular,
  hence all prerequisites of the Implicit Function Theorem are
  satisfied and the claim is proved.

  It remains to consider the situation when one of the operators is
  zero.  This is certainly not $T_s$, since $T_s\phi=0$ for all $\phi
  \in C^\infty_0(I)$ implies $\sin u=0$ and, due to $|u|<\pi$, $u
  \equiv 0$ in $I$.  There is, however, a valid situation such that
  $T_c=0$, and that is if and only if $u(t)=\pi/2\sign t$ is piecewise
  constant with a jump of height $\pi$.  But then the second component
  of $P(\eps,0,\beta)$ is zero for all $\eps\geq0$, all $\beta \geq 0$
  and any anti-symmetric function $g$.  Fixing such $g$ with $T_s g
  \not=0$, we can thus apply the Implicit Function Theorem to $\tilde
  P(\eps,\beta)=P_1(\eps,0,\beta)$ and get the desired result with
  $f=\alpha_\eps=0$.
\end{proof}

So we have found an approximation $u_\eps + \alpha_\eps f + \beta_\eps
g$ on $I$ that satisfies the constraints. Due to $\alpha_0=\beta_0=0$,
$\alpha_\eps$ and $\beta_\eps$ converge to zero as $\eps\to0$; then
$u_\eps + \alpha_\eps f + \beta_\eps g \to u$ in $L^1(I)$ and $\hat
q_\eps$, defined by $u_\eps + \alpha_\eps f + \beta_\eps g$ and the
endpoints of $q(I)$, converges to $q$ in $H^1(I;\R^2)$ and satisfies
\begin{equation}
  \label{eq:curve_recovery}
  \lim_{\eps\to0} \F(\hat q_\eps,0,I_\eps) = \hat\sigma |[q']|.
\end{equation}
As the number of kinks is finite, the construction can be made for
each $s \in S_q$ on an interval $I^s_\eps$.

% - Phase field ---------------------------------------------------------------

\subsubsection{Phase field recovery and energy estimates}

Now we construct a recovery sequence for the phase field which is in
line with $q_\eps$ from above.  It is well-known, see for example
\cite{A00}, that the optimal profile for a transition of
$v_\eps$ from $-1$ to $+1$ is obtained by minimising
\begin{equation*}
  G_\eps(v) = \int_\R \eps |v'|^2 + \tfrac{1}{\eps} \Phi(v) \,dt
\end{equation*}
among functions $v$ that satisfy $v(0)=0$ and $v(\pm\infty)=\pm1$.
Indeed, setting $v_\eps(t)=v(t/\eps)$ we observe
\begin{equation*}
  G_\eps(v_\eps) = G_1(v)
  \geq 2 \int_\R \sqrt{\Phi(v)}v' \,dt
  = 2 \int_{-1}^1 \sqrt{\Phi(v)} \,dv
  = \sigma.
\end{equation*}
Equality holds if and only if $v'=\sqrt{\Phi(v)}$, which admits a
local solution $p$ with initial condition $p(0)=0$ because
$\sqrt{\Phi(\cdot)}$ is continuous. Obviously the constants $+1$ and
$-1$ are a global super- and sub-solution of the problem, hence $p$
can be extended to the whole real line.  Since $\Phi(p)>0$ for
$p\in(-1,+1)$, $p(t)$ converges to $\pm1$ as $t\to\pm\infty$. Thus
$p(t/\eps)$ minimises $G_\eps$. Due to the symmetry of $\Phi$ we can
presume $-p(-t)=p(t)$ and need to know the profile only for $t\geq0$.

We assume again, that by identification with appropriate coordinates
the (ghost) interface is located at $s=0$ and that the phase field is
given on an interval $I$ containing $s$. The building block for the
recovery sequence is
\begin{equation*}
  p_\eps(t) =
  \begin{cases}
    0
    & : 0 \leq t \leq \delta_\eps, \\
    p\left( \frac{t-\delta_\eps}{\eps} \right)
    & : \delta_\eps < t \leq \delta_\eps+\sqrt{\eps}, \\
    p(1/\sqrt{\eps}) + \tfrac{1}{\eps} \left(
      t-\delta_\eps-\sqrt{\eps} \right)
    & : \delta_\eps+\sqrt{\eps} < t \leq \delta_\eps+\sqrt{\eps}
    + \eps \left(
      1-p(1/\sqrt{\eps}) \right), \\
    1
    & : \delta_\eps+\sqrt{\eps} + \eps \left(
      1-p(1/\sqrt{\eps}) \right) < t,
  \end{cases}
\end{equation*}
which connects $p_\eps=0$ and $p_\eps=1$ by a transition according to
the optimal profile and a linear function, see Figure \ref{fig:p_eps};
the length of $\{p_\eps=0\}$ is chosen consistently with the recovery
of the curve, that is $\delta_\eps=|[q'](0)|/(2\sqrt{\Phi(0)}) \cdot
\eps$. In the next lemma we estimate the interface energy of the
nonzero part of $p_\eps$.

\begin{figure}
  \centering
  \includegraphics{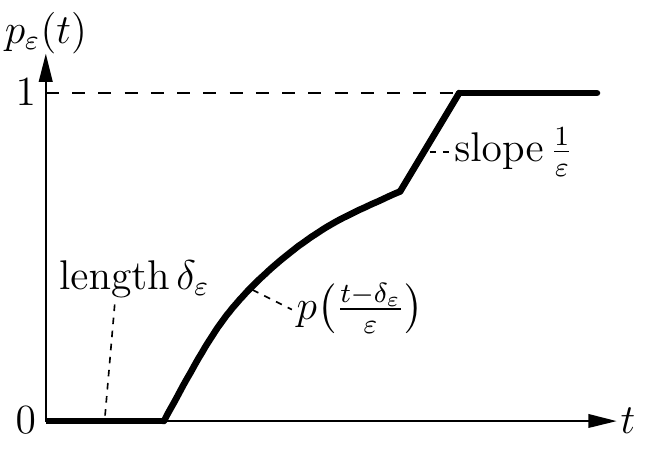}
  \caption{Construction of $p_\eps$, consisting of space for the curve
    recovery, the optimal profile and the connection to $1$.}
  \label{fig:p_eps}
\end{figure}

\begin{lemma}
  \label{p_eps_energy_limit}
  For any curve $q \in H^2(I\sm\{0\};\R^2)$ we have
  \begin{equation*}
    \limsup_{\eps\to0} \F_\eps(q,p_\eps,I\cap\{t>\delta_\eps\})
    \leq \tfrac{\sigma}{2}.
  \end{equation*}
\end{lemma}

\begin{proof}
  The curve satisfies $\kappa_q \in L^2(I)$, so the integral of $\eps
  \kappa^2$ vanishes as $\eps\to0$.  The other terms are for
  sufficiently small $\eps$ easily estimated by
  \begin{equation*}
    \int_{I\cap\{t>\delta_\eps\}}
    \eps |p_\eps'|^2 + \tfrac{1}{\eps} \Phi(p_\eps) \,dr
    \leq
    \int_{0}^{1/\sqrt{\eps}}
    |p'(r)|^2 + \Phi(p(r)) \,dr
    + \left(1-p\left(1/\sqrt{\eps}\right)\right)
    ( 1 + \sup_{[0,1]} \Phi ).
  \end{equation*}
  Taking the upper limit $\eps\to0$ and recalling the symmetry of $\Phi$
  yield the result.
\end{proof}

It is now evident how to construct the recovery sequence for an
interface: we simply take
\begin{equation*}
  v_\eps(t)=
  \begin{cases}
    p_\eps(t)   & : 0 \leq t, \\
    -p_\eps(-t) & : 0 > t
  \end{cases}
\end{equation*}
or the negative of it. Then $v_\eps$ converges to $v$ in $L^1(I)$; it
follows from Lemma \ref{p_eps_energy_limit} that these
approximations demonstrate the desired energy behaviour; in
addition the volume constraint is not affected by substituting
$v_\eps$ for $v$ due to symmetry of $p_\eps$.

In the case of a ghost interface we use the combination of $p_\eps(t)$
and $p_\eps(-t)$ or its negative. Again, convergence and energy
behaviour are as required, but the volume constraint is violated.  To
get an admissible recovery sequence we add a small correction as we
did with the angle function. Let $h:I \to \R$ be smooth, have compact
support in $I \cap \R_{>0}$ and satisfy $\int_I h \,dt=1$. Then the
volume constraint is satisfied by $v_\eps+\gamma_\eps h$, if
\begin{equation*}
  \gamma_\eps = \int_I v-v_\eps \,dt.
\end{equation*}
Since
\begin{equation*}
  \int_{\delta_\eps}^{\delta_\eps+\sqrt{\eps}} 1-p_\eps \,dt
  = \sqrt{\eps}\int_0^1 1 - p(t/\sqrt{\eps}) \,dt
  = o(\sqrt{\eps}),
\end{equation*}
$\gamma_\eps$ is of order $o(\sqrt{\eps})$, too.  This is enough to
still ensure convergence $v_\eps+\gamma_\eps h \to v$ in $L^1(I)$ and
the energy inequality
\begin{equation}
  \label{eq:phase_field_limsup}
  \limsup_{\eps\to\infty}
  \F_\eps(q,p_\eps+\gamma_\eps h,I \cap \{t>\delta_\eps\}))
  \leq
  \tfrac{\sigma}{2},
\end{equation}
thanks to
\begin{equation*}
  \tfrac{1}{\eps} \Phi(\pm1+\gamma_\eps h)
  = \tfrac{1}{\eps} \left(
    \Phi(\pm1) + \gamma_\eps h \Phi'(\pm1) +
    O(\gamma^2_\eps) \right)
  = o(1).
\end{equation*}
Therefore $v$ can be recovered around each interface, and the sequence for
$v$ on $\S^1$ is now built by substituting $v_\eps+\gamma_\eps h$
locally for $v$.  Combining the constructions for phase field and
curve we have the following result.

\begin{corollary}
  For $(q,v) \in \C \times \P$ and sufficiently
  small $\eps$ there is $(q_\eps,v_\eps) \in \C_\eps \times
  \P_\eps$ such that $q_\eps \to q$ in $H^1(\S^1;\R^2)$,
  $v_\eps \to v$ in $L^1(\S^1)$ and
  \begin{equation*}
    \limsup_{\eps\to0} \E_\eps(q_\eps,v_\eps) \leq \E(q,v).
  \end{equation*}
\end{corollary}

\begin{proof}
  Denote by $q_\eps$ and $v_\eps$ the recovery sequences from this
  subsection. The convergence results have already been established;
  for the inequality note that for each $s \in S$ we have intervals
  $I^s_\eps \subset I^s$ such that the kink is smoothed out in
  $I^s_\eps$ and the phase field transition is made in $I^s \sm
  I^s_\eps$; so combining the estimates \eqref{eq:curve_recovery} and
  \eqref{eq:phase_field_limsup} we get
  \begin{equation}
    \label{eq:limsup_Is}
    \limsup_{\eps\to0} \F_\eps(q_\eps,v_\eps,I^s)
    =
    \limsup_{\eps\to0} \big( \F_\eps(q_\eps,0,I^s_\eps) +
      \F_\eps(q,v_\eps,I^s \sm I^s_\eps) \big)
    \leq
    \F(q,v,I^s).
  \end{equation}
  Outside $J:=\bigcup_{s \in S} I^s$ phase field and curve remain
  unchanged, therefore
  \begin{equation*}
    \limsup_{\eps\to0} \F_\eps(q_\eps,v_\eps,\S^1 \sm J)
    =
    \limsup_{\eps\to0}\F_\eps(q,v,\S^1 \sm J)
    =
    0.
  \end{equation*}
  Together with \eqref{eq:limsup_Is} summed over all $s \in S$ this
  yields the upper bound inequality for $\F_\eps(q_\eps,v_\eps,\S^1)$
  and $\F(q,v,\S^1)$. Finally, because $v_\eps$ is zero where
  $\kappa_{q_\eps}$ differs from $\kappa_q$ by more than the small
  correction for the endpoint constraints, we have
  \begin{align*}
    \int_{\S^1} v_\eps^2 \left(
      \kappa_{q_\eps} - C(v_\eps) \right)^2 dt
    &\leq
    \int_{\S^1 \sm \cup I^s_\eps} \left(
      \kappa_q - C(v_\eps) \right)^2 dt + o(1) \\
    &\leq
    \int_{\S^1 \sm S} \left(
      \kappa_q - C(v_\eps) \right)^2 dt + o(1),
  \end{align*}
  and taking the upper limit finishes the proof.
\end{proof}

% -----------------------------------------------------------------------------
% -
% - Acknowledgements
% -
% -----------------------------------------------------------------------------

\section*{Acknowledgement}

This work was supported by the EPSRC Science and Innovation award to
the Oxford Centre for Nonlinear PDE (EP/E035027/1).
The author thanks Barbara Niethammer and Michael Herrmann for their
support, encouragement and many stimulating discussions.
Finally, the author wishes to thank the referees whose careful and
competent work improved the exposition to a great extent.

% -----------------------------------------------------------------------------
% -
% - References
% -
% -----------------------------------------------------------------------------

\bibliography{snapping_curves}
\bibliographystyle{abbrv}

\end{document}